\documentclass[12pt,reqno]{amsart}
\usepackage{amsmath}
\usepackage{amssymb}
\usepackage{amsthm}
\usepackage{enumerate}
\usepackage{mathrsfs} 
\usepackage{array}
\usepackage{csquotes}
\theoremstyle{plain}
\newtheorem{theorem}{Theorem}[section]

\newtheorem{corollary}[theorem]{Corollary}

\theoremstyle{definition}
\newtheorem{definition}{Definition}
\newtheorem{remark}{\textup{Remark}}

\usepackage{enumitem}
\usepackage[colorlinks,citecolor=blue]{hyperref}
\begin{document}

\title[On Horadam-Lucas sequence]{On Horadam-Lucas sequence}
\author[A.\,Da\c{s}dem\.{I}r]%
{Ahmet Da\c{s}dem\.{I}r}

\newcommand{\acr}{\newline\indent}

\address{\llap{*\,}Department of Mathematics\acr
                   Faculty of Arts and Sciences\acr
                   Kastamonu University\acr
                   37150 Kastamonu\acr
                   TURKEY}

\email{ahmetdasdemir37@gmail.com}



\begin{abstract}
Horadam introduced a new generalized sequence of numbers, describing its key features and the special sub-sequences that are obtained depending on the choices of initial parameters. This sequence and its sub-sequences are known as the Horadam, generalized Fibonacci, and generalized Lucas numbers, respectively. In the present study, we propose another new sequence, which satisfies a second-order recurrence relation, in addition to Horadam's definition. Further, we prove the Binet's formula, some famous identities, and summation formulas for this new sequence. In particular, we demonstrate the interrelationships between our new sequence and the Horadam sequence.
\end{abstract}

\maketitle

\section{Introduction}\label{intro}
In \cite{hor1}, Horadam considered a generalized form of the classic Fibonacci numbers, changing the initial terms $ F_0=0 $ and $ F_1=1 $ to $ a $ and $ b $, respectively. Then  in \cite{hor2}, Horadam defined the second-order linear recurrence sequence $ w_n\left(a,b;p,q\right) $:
\begin{equation}\label{I1}
w_n = pw_{n-1}-qw_{n-2}
\end{equation}
with $ w_0 = a $ and $ w_1 = b $. This generalizes many sequences of integers; e.g., the Fibonacci, Lucas, Pell, Pell-Lucas, Jacobsthal, Jacobsthal-Lucas,  Generalized Fibonacci and  Generalized Lucas sequences. In addition, the Binet's formula for the Horadam sequence is 
\begin{equation}\label{I2}
w_n=\frac{A\alpha^n-B\beta^n}{\alpha-\beta}.
\end{equation}
In \eqref{I2}, the author has used the notations
\begin{equation}\label{I3}
\alpha=\frac{p+d}{2},\,\beta=\frac{p-d}{2},\,A = b-a{\beta},\,B = b-a{\alpha}\,\text{and}\,d=\sqrt{p^2-4q}.
\end{equation}
One can readily show that
\begin{equation}\label{I4}
\alpha+\beta=p,\,\,\alpha\beta=q,\,\,\alpha-\beta=d,
\end{equation}
\begin{equation}\label{I5}
A+B=2b-ap,\,\,A-B=ad,\,\,\text{and}\,\,AB=b^2-abp+a^2q=E.
\end{equation}
In working with this sequence, it is useful to consider the following special cases:
\begin{align}
\label{I6}w_n\left(0,1;p,q\right)=u_n\left(p,q\right)\\
\label{I7}w_n\left(2,p;p,q\right)=v_n\left(p,q\right).
\end{align}
It is notable that, for these special cases, the Binet's formula in \eqref{I2} reduces to 
\begin{align}
\label{I8}u_n=\frac{\alpha^{n}-\beta^{n}}{d}\\
\label{I9}v_n=\alpha^n+\beta^n.
\end{align}

For positive integers $ n $, Horadam has given the following formulas for $ w_n $ , $ u_n $ and $ v_n $ \cite{hor3}:
\begin{equation}\label{I10}
w_{-n}=q^{-n}\frac{au_n-bu_{n-1}}{au_n+\left(b-pa\right)u_{n-1}},\,u_{-n}=q^{-n+1}u_{n-2}\,\text{and}\,v_{-n}=q^{-n}v_{n}.
\end{equation}

In \cite{hor4}, Horadam presented some geometrical interpretations of the Horadam sequence, including the Pythagorean property. In \cite{CM}, Morales defined the $ 2 \times 2 $ matrix
\begin{equation}\label{I11}
U\left(p,q \right)=\left[ {\begin{array}{*{20}{c}}
	{p}&{-q} \\ 
	{1}&{0} 
	\end{array}} \right]
\end{equation}
and showed that
\begin{equation}\label{I12}
U^n\left(p,q \right)=\left[ {\begin{array}{*{20}{c}}
	{{u}_{n+1}}&{-q{u}_{n}} \\ 
	{{u}_{n}}&{-q{u}_{n-1}} 
	\end{array}} \right].
\end{equation}
For brevity, we denote the matrix  $ U\left(p,q \right) $ by $ R $ unless stated otherwise. Note that Larcombe et al. have presented well-known
systematic investigations of the Horadam sequence in \cite{LBF}.

In this paper, we define a new generalization $ h_n\left(a,b;p,q\right) $ of the second-order linear sequences, i.e., Fibonacci, Lucas, Pell, Jacobsthal, Generalized Fibonacci, and Generalized Lucas sequences. We present many results from this new generalization, including the Binet's formula, the d'Ocagne's and Gelin-Ces\'aro identities and some sum formulas. Further, we give some special identities of $ h_n\left(a,b;p,q\right) $ via matrix techniques.

\section{Main Results}

Imagine that we study on the second-order sequence given in Eq. \eqref{I1}. It is well-known that the Horadam sequence generalizes certain integer sequences, depending on the choice of the initial parameters $ a $, $ b $, $ p $, and $ q $. There is also another generalized sequence of integers having the same recurrence relation. We give our generalized sequence as follows:

\begin{theorem}[Binet's formula]\label{D}
	For every integer $ n $, we have the Binet's formula
	\begin{equation}\label{4}
		{h_{n}}=A\alpha^n+B\beta^n,
	\end{equation}
	where $ A =  b-a{\beta} $ and $ B = b-a{\alpha} $, as given by Horadam \cite{hor2}.
\end{theorem}
\begin{proof}
	Equation \eqref{1} is a second-order linear homogeneous difference equation, with constant coefficients, which has the form
	\begin{equation}\label{5}
		{x_{n}} = p{x_{n-1}} - q{x_{n - 2}}.
	\end{equation}
	We suppose that there is a solution to Eq. \eqref{1} of the form
	\begin{equation}\label{6}
		{x_{n}} = \lambda ^ n,
	\end{equation}
	where $ \lambda $ is a constant to be determined. Substituting Eq. \eqref{6} into Eq. \eqref{5} yields
	\begin{equation*}
		{\lambda^{n}} = p{\lambda^{n-1}} - q{\lambda^{n - 2}}
	\end{equation*}
	or more simplify
	\begin{equation}\label{7}
		{\lambda^{2}} - p{\lambda} + q = 0,
	\end{equation}
	which was given by Horadam \cite{hor3}. The roots of Eq. \eqref{7} are
	\begin{equation}\label{8}
		{\lambda_{1}}=\left(p+d\right)/2={\alpha}\,\,\text{and}\,\,{\lambda_{2}}=\left(p-d\right)/2={\beta}.
	\end{equation}
	As a result, we have found two independent solutions to Eq. \eqref{5} in the form of Eq. \eqref{6}. Hence, a linear combination of these two solutions is also a solution of Eq. \eqref{1}, namely  
	\begin{equation}\label{9}
		{h_{n}}=c_1\alpha^n+c_2\beta^n.
	\end{equation}
	Considering the initial conditions for our definition, we can write
	\begin{equation}\label{10}
		\begin{gathered}
			c_1+c_2=2b-ap \hfill \\
			c_1\alpha+c_2\beta=bp-2aq.
		\end{gathered}
	\end{equation}
	If the system of equations in \eqref{10} is solved, we obtain the two solutions 
	\begin{equation}\label{11}
		{c_{1}} = b-a{\beta}\,\,\text{and}\,\,{c_{2}}=b-a{\alpha},
	\end{equation}
	and the result follows.
\end{proof}

\begin{definition}\label{A}
Let $ n $ be any integer. Then, for $ n \geqslant 2 $, we define
\begin{equation}\label{1}
{h_{n}} = p{h_{n-1}} - q{h_{n - 2}},
\end{equation}
with the initial conditions ${h_0} = 2b - ap$ and $ {h_1} = bp - 2aq $.
\end{definition}

We term the sequence in \eqref{1} a Horadam-Lucas sequence. Table \ref{tab:1} presents the first five terms of this sequence. Particular cases of the Horadam-Lucas sequence are as follows:
\begin{itemize}
	\item For $ a = 0 $, $ b = p $, we can write
	\begin{equation*}
	{h_{n}}\left(0,p;p,q\right)=A\alpha^n+B\beta^n=v_{n}.
	\end{equation*}
	
	\item For $ a = 2 $, $ b = p $, we can write
	\begin{equation*}
	{h_{n}}\left(2,p;p,q\right)=A\alpha^n+B\beta^n=d^2{u}_{n}.
	\end{equation*}
\end{itemize}

In order to investigate any integer sequences from Eq. \eqref{I1}, we must consider different values of $ a $, $ b $, $ p $, and $ q $. However, the definitions in \eqref{I1} and \eqref{1} allow us to investigate, respectively, the primary sequence i.e., Fibonacci and Pell numbers and the secondary sequences i.e., the Lucas and Pell-Lucas numbers at the same time. Briefly, then $ {h_{n}} $ may be considered to be a companion sequence to $ {w_{n}} $. This is summarized in Table \ref{tab:2}.

Eq. \eqref{1} is not any sequence but a special case. It has quite amazing features. Obviously, we presented this definition by inspiring from the following theorem, which contains the Binet's formula and will be used broadly.

\begin{table}[t!]
	\caption{\label{tab:1} The first five terms of $ w_n $ and $ h_n $.}
	\centering
	\begin{tabular}{|l|c|c|}	\hline
		n & Horadam numbers                                                & Horadam-Lucas numbers    \\\hline\raisebox{15pt}
		
		0 & \footnotesize{$ a $}                                                  & \footnotesize{$2b - ap$}                 \\\raisebox{15pt}
		
		1 & \footnotesize{$ b $}                                                  & \footnotesize{$bp - 2aq$}                          \\\raisebox{15pt}
		
		2 & \footnotesize{$ bp-aq $}                                              & \footnotesize{$b \left( p^2-2q\right)-apq$}     \\\raisebox{15pt}
		
		3 &\footnotesize{$ b \left( p^2-q\right)-apq $}                           & \footnotesize{$bp\left(p^2-3q\right)-aq\left(p^2-2q\right)$} \\\raisebox{15pt}
		
		4 &\footnotesize{$ bp \left( p^2-2q\right) -aq \left( p^2-q\right) $}     & \footnotesize{$ b\left( p^4-4p^2q+2q^2\right) -ap\left( p^2-3q\right) $}      \\\raisebox{15pt}
		
		5 &\footnotesize{$ b\left( p^4-3p^2q+q^2\right) -ap\left( p^2-2q\right) $}& \footnotesize{$bp\left( p^4-5p^2q+5q^2\right)-aq\left( p^4-4p^2q+2q^2\right)$}\\\hline 
	\end{tabular}
\end{table}
\begin{table}[h!]
	\caption{\label{tab:2} Sequences corresponding to different choices of $ a $, $ b $, $ p $, and $ q $.}
	\centering
	\begin{tabular}{|l|c|c|c|c|c|}	\hline
		a  &  b  &  p  &  q &  Horadam sequence        &  Horadam-Lucas sequence   \\\hline\raisebox{15pt}
		
		0  &  1  &  1  & -1 &  Fibonacci numbers       &  Lucas numbers             \\\raisebox{15pt}
		
		0  &  1  &  2  & -1 &  Pell numbers            &  Pell-Lucas numbers        \\\raisebox{15pt}
		
		0  &  1  &  1  & -2 &  Jacobsthal numbers      &  Jacobsthal-Lucas numbers  \\\hline 
	\end{tabular}
\end{table}

This gives us the following theorem.
\begin{theorem}[De Moivre's Formula]\label{DD}
	Let $ x_n  = h_{n+1}-qh_{n -1}$ and $ k $ be any integer. Then, we have
	\begin{equation}\label{1111}
	\left( \frac{x_n + h_n d}{2Ad} \right)^k = \frac{x_{kn} +h_{kn} d}{2Ad}.
	\end{equation}
\end{theorem}
\begin{proof}
	From  Eq. \eqref{4}, we can write
	\begin{align*}
	&{h_{n + 1}}=\left( A\alpha\right) \alpha^n+\left( B\beta \right) \beta^n \\
	&{h_{n}}=A\alpha^n+B\beta^n.
	\end{align*}
	Solving this system of equations permits us to obtain
	\begin{equation*}
	\alpha^n = \frac{h_{n + 1}- \beta h_{n}}{Ad}\,\,\text{and}\,\,\beta^n = - \frac{h_{n + 1}- \alpha h_{n}}{Bd}.
	\end{equation*}
	Since $ \left( \alpha^n \right)^k = \alpha^{\left(nk\right)} $, with a little computation, the result follows.
\end{proof}
Note that Eq \eqref{1111} has a similar form with the famous de Moivre's formula.

From Definition \ref{A}, we can immediately obtain the following result:

\begin{corollary}[Brahmagupta identity]\label{B}
Let $ n $, $ m $, $ r $, $ s $, $ k $, and $ t $ be any integers. Then, we have
\begin{equation}\label{2}
\left( {k{h_n}^2 + t{h_m}^2} \right)\left( {k{h_r}^2 + t{h_s}^2} \right) = {\left( {k{h_n}{h_r} - t{h_m}{h_s}} \right)^2} + kt{\left( {{h_n}{h_s} + {h_m}{h_r}} \right)^2}.
\end{equation}
\end{corollary}
\begin{proof}
The proof can be completed readily by means of some algebraic operations.	
\end{proof}

\begin{remark}\label{Y1}
We can interpret Corollary \ref{B} as follows. The product of two Horadam-Lucas numbers that is a linear combination of the sum of any two Horadam-Lucas numbers is the sum of two squares.
\end{remark}

\begin{theorem}[Universal recurrence]\label{Y2}
For any psoitive integer $ n $, the Horadam-Lucas sequence satisfies the recurrence relation
\begin{equation}\label{221}
{h_{n + 3}} = \frac{{h_{n + 1}^3} -2 {h_{n}}{h_{n+1}}{h_{n+2}} +{h_{n - 1}}{h_{n + 2}^2} }{{h_{n-1}{h_{n+1}} -{h_{n}^2}}}.
\end{equation}
\end{theorem}
\begin{proof}
Solving the system of equations to be obtained after writing the integers $ n $ and $ n + 1 $ in \eqref{1}, with respect to $ p $ and $ q $, the result follows.
\end{proof}
We call Eq. \eqref{221} the \enquote{Universal Recurrence} since it is a second-order recurrence relation that has independent of the choice $ p $, $ q $, $ a $, and $ b $.

Next we prove a theorem that demonstrates the interrelationships between the Horadam and Horadam-Lucas numbers.
\begin{theorem}\label{C}
Let $ n $ be any integer. Then, we have
\begin{equation}\label{3}
h_{n}=w_{n+1}-qw_{n-1}
\end{equation}
and
\begin{equation}\label{311}
w_{n}=\frac{2h_{n+1}-ph_{n-1}}{d^2}.
\end{equation}
\end{theorem}
\begin{proof}
Use the induction method on $ n $. For $ n = 1 $ and $ n = 2 $, the validity of Eq. \eqref{3} is trivial. We assume that Eq. \eqref{3} is true for the first preceding terms $ n $. Then we can write
\begin{align*}
h_{n+1}&=ph_{n}-qh_{n-1}=p\left(w_{n+1}-qw_{n-1}\right)-q\left(w_{n}-qw_{n-2}\right)\\
&=\left(pw_{n+1}-qw_{n}\right)-q\left(pw_{n-1}-qw_{n-2}\right)\\
&=w_{n+2}-qw_{n},
\end{align*}
which is the desired result. The validity of Eq. \eqref{311} can be shown using Eq. \eqref{3}.
\end{proof}

From Eq. \eqref{8}, Eq. \eqref{7} has two distinct roots, i.e., $ \alpha $ and $ \beta $. Therefore, the Horadam-Lucas sequence can be extended to negative subscripts $ n $. For combinatorial simplicity, we denote the Horadam-Lucas numbers with negative subscripts by $ \overline{h}_n $. We then have the following recurrence relations: For $ n > 0 $,
\begin{equation}\label{12}
\overline{h}_0 = 2b-ap,\,\,\overline{h}_1=\frac{bp-a\left(p^2-2q\right)}{q}\,\,\text{and}\,\,{\overline{h}_{n + 1}} = \frac{p}{q}{\overline{h}_n} - \frac{1}{q}{\overline{h}_{n - 1}}.
\end{equation}

Based on Eq. \eqref{4}, we can also prove the following results:
\begin{theorem}
Let $ n $ and $ m $ be any integers. Then, we have
\begin{equation}\label{111}
{h_{n}}^2-d^2{w_n}^2 =4Eq^n,
\end{equation}
\begin{equation}\label{112}
{w_{m+n}}+q^n{w_{m-n}} ={w_{m}}{v_{n}},
\end{equation}
\begin{equation}\label{113}
{w_{m+n}}-q^n{w_{m-n}} ={h_{m}}{u_{n}},
\end{equation}
\begin{equation}\label{114}
{h_{m+n}}+q^n{h_{m-n}} ={h_{m}}{v_{n}},
\end{equation}
\begin{equation}\label{115}
{h_{m+n}}-q^n{h_{m-n}} =d^2{w_{m}}{u_{n}},
\end{equation}
\begin{equation}\label{116}
2{w_{m+n}}={w_{m}}{v_{n}}+{v_{m}}{w_{n}},
\end{equation}
\begin{equation}\label{117}
2q^n{w_{m-n}}={w_{m}}{v_{n}}-{v_{m}}{w_{n}},
\end{equation}
\begin{equation}\label{118}
2{h_{m+n}}={h_{m}}\left({v_{n}}+d^2{u_{n}}\right),
\end{equation}
and
\begin{equation}\label{119}
2q^n{h_{m+n}}={h_{m}}\left({v_{n}}-d^2{u_{n}}\right).
\end{equation}
\end{theorem}
\begin{proof}
Using Eq. \eqref{4}, we can write
\begin{align*}
{h_{n}}^2-d^2{w_n}^2&=\left(A\alpha^n+B\beta^n\right)^2-d^2\left(\frac{A\alpha^n-B\beta^n}{\alpha-\beta}\right)^2\\
&=4AB\alpha^n\beta^n,
\end{align*}
and the result follows. We can also prove Eqs. \eqref{112}-\eqref{115} by repeating the same operations. On the other hand, Eqs. \eqref{116}-\eqref{119} are proved by summing and subtracting the appropriate equations from Eqs. \eqref{112}-\eqref{115}.

\end{proof}

\section{Special Properties of $ {h}_n $}

In this section, we present some special properties of the generalized sequence defined in \eqref{1}. We first define the generating function given in the
form
\begin{equation}\label{13}
h\left( {x} \right) = \sum\limits_{n = 0}^\infty  {{h_n}{x^n}}.
\end{equation}
Then we can write the following theorem.
\begin{theorem}\label{E}
The generating function and the exponential generating function of the Horadam-Lucas sequence are given by
\begin{equation}\label{14}
h\left( {x} \right)=\frac{{h_0}+\left({h_1} - p{h_0}\right)x}{1 -px-q x^2}
\end{equation}
and
\begin{equation}\label{15}
h\left( {x} \right)=A e^{\alpha x}+Be^{\beta x},
\end{equation}
respectively.
\end{theorem}
\begin{proof}
Summing these equations after setting up $ h\left( x \right) $, $ -pxh\left( x \right) $ and $ -q x^2 h\left( x \right) $ readily yields the first result. Substituting Binet's formula from \eqref{4} into the generating function representation \eqref{13}, and considering the MacLaurin series for an exponential function, then yields the second result.
\end{proof}

\begin{theorem}[Pythagorean formula]\label{F} Let $ n $ be any integer. Then, we have
\begin{align}
\nonumber\left(\frac{p}{q^2}h_{n}h_{n+3}\right)^2+&\left(2Ph_{n+2}\left(2Ph_{n+2}-h_{n}\right)\right)^2 \\
&\label{16}\hspace{2cm}=\left(h_{n}^2+2Ph_{n+2}\left(2Ph_{n+2}-h_{n}\right)\right)^2,
\end{align}
where $ P = \frac{p^2-q}{2q^2} $.
\end{theorem}
\begin{proof}
Using Eq. \eqref{1}, we can write
\begin{align*}
&\left(p^2-q\right)h_{n+2}-ph_{n+3}=q^2h_{n}\\
&\left(p^2-q\right)h_{n+2}+ph_{n+3}=2\left(p^2-q\right)h_{n+2}-q^2h_{n}.
\end{align*}
Multiplying these equations side-by-side, we obtain
\begin{equation*}
\left(p^2-q\right)^2h_{n+2}^2-p^2h_{n+3}^2=2q^2\left(p^2-q\right)h_{n}h_{n+2}-q^4h_{n}^2
\end{equation*}
and we rearrange it to obtain
\begin{equation*}
\left(ph_{n+3}\right)^2=\left(\left(p^2-q\right)h_{n+2}\right)^2-2q^2\left(p^2-q\right)h_{n}h_{n+2}+\left(q^2h_{n}\right)^2.
\end{equation*}
Dividing by $ q^2 $ after multiplying the last equation by $ h_n^2 $ and adding $ \left(p^2-q\right)h_{n+2}\left(\left(p^2-q\right)h_{n+2}-2q^2h_n\right) $ to each side, we obtain the claimed result.
\end{proof}

From Theorem \ref{F}, we also obtain the following result.

\begin{corollary}\label{G}
All Pythagorean triples can be generated in terms of Horadam-Lucas numbers.
\end{corollary}

\begin{theorem}\label{H}
For any integer $ n $, we have
\begin{equation}\label{17}
{h_{n + 1}}^2 - q{h_n}^2 = {d^2}\left[ {\left( {{b^2} - {a^2}q} \right){{u}_{2n + 1}} - aq\left( {2b - ap} \right){{u}_{2n}}} \right].
\end{equation}
\end{theorem}
\begin{proof}
To prove this property, we use the Binet's formula for $ h_{n} $.
\begin{align}
\nonumber{h_{n + 1}}^2 - q{h_n}^2&=\left(A\alpha^{n+1}+B\beta^{n+1}\right)^2-q\left(A\alpha^{n}+B\beta^{n}\right)^2\\
\nonumber&=A^2\alpha^{2n+2}+B^2\beta^{2n+2}+2E\alpha^{n+1}\beta^{n+1}\\
\nonumber&\hspace{3cm}-q\left(A^2\alpha^{2n}+B^2\beta^{2n}+2E\alpha^{n}\beta^{n}\right)\\
\nonumber&=A^2\alpha^{2n+2}+B^2\beta^{2n+2}+2Eq^{n+1}\\
\nonumber&\hspace{3.5cm}-q\left(A^2\alpha^{2n}+B^2\beta^{2n}+2Eq^{n}\right)\\
\nonumber&=A^2\alpha^{2n+2}+B^2\beta^{2n+2}-qA^2\alpha^{2n}-qB^2\beta^{2n}\\
\label{18}&=A^2\left(\alpha^{2}-q\right)\alpha^{2n}+B^2\left(\beta^{2}-q\right)\beta^{2n}
\end{align}
Here after some mathematical operations, we can write
\begin{equation*}
A^2\left(\alpha^{2}-q\right)=d\left[\left({{b^2} - {a^2}q}\right)\alpha- aq\left( {2b - ap} \right)\right]
\end{equation*}
and
\begin{equation*}
B^2\left(\beta^{2}-q\right)=-d\left[\left({{b^2} - {a^2}q}\right)\beta- aq\left( {2b - ap} \right)\right].
\end{equation*}
Substituting these equations into Eq. \eqref{18} yields the desired result.
\end{proof}

We next prove the following important theorem, which we shall reduce to obtain a number of special identities.
\begin{theorem}[Vajda's identity]\label{I}
Let $ n $, $ r $, and $ s $ be any integers. Then,
\begin{equation}\label{19}
{h_{n+s}}{h_{n-r}}-{h_{n}}{h_{n-r+s}}=Eq^{n-r}\left(v_{r+s}-q^{s}v_{r-s}\right).
\end{equation}
\end{theorem}
\begin{proof}
	\begin{align*}
	{h_{n}}{h_{n-r+s}}-&{h_{n+s}}{h_{n-r}}=\left(A\alpha^{n}+B\beta^{n}\right)\left(A\alpha^{n-r+s}+B\beta^{n-r+s}\right)\\
	&\hspace{3.5cm}-\left(A\alpha^{n+s}+B\beta^{n+s}\right)\left(A\alpha^{n-r}+B\beta^{n-r}\right)\\
	&\hspace{0.5cm}-A^2\alpha^{2n-r+s}-E\alpha^{n-r}\beta^{n+s}-E\alpha^{n+s}\beta^{n-r}-B^2\beta^{2n-r+s}\\
	&=E\left(\alpha^{n}\beta^{n-r+s}+\alpha^{n-r+s}\beta^{n}-\alpha^{n-r}\beta^{n+s}-\alpha^{n+s}\beta^{n-r}\right)\\
	&=E\left(\alpha^{n-r+s}\beta^{n-r+s}\left(\alpha^{r-s}+\beta^{r-s}\right)-\alpha^{n-r}\beta^{n-r}\left(\alpha^{r+s}+\beta^{r+s}\right)\right)\\
	&=E\left(q^{n-r+s}v_{r-s}-q^{n-r}v_{r+s}\right)\\
	&=-Eq^{n-r}\left(v_{r+s}-q^{s}v_{r-s}\right),
	\end{align*}
	which is the claimed result.
\end{proof}
From Vajda's identity, we also have the following special identities:
\begin{itemize}
	\item For $ r = s $, we find the Catalan's identity: 
	\begin{equation}\label{20}
	{h_{n+r}}{h_{n-r}}-{h_{n}^2}=Eq^{n-r}\left(v_{2r}-2q^r\right)
	\end{equation}
	
	\item For $ r = s =1 $, we find the Cassini's identity: 
	\begin{equation}\label{21}
	{h_{n+1}}{h_{n-1}}-{h_{n}^2}=Ed^2q^{n-1}
	\end{equation}
	
	\item For $ n - r = m $ and $ s = 1 $, we recover the d'Ocagne's identity: 
	\begin{equation}\label{22}
	{h_{m}}{h_{n+1}}-{h_{n}}{h_{m+1}}=Eq^{m}\left(v_{n-m+1}-qv_{n-m-1}\right)
	\end{equation}
\end{itemize}
In addition, we can prove the following theorem.
\begin{theorem}[Gelin-Ces\'aro identity]\label{J} For any integer $ n $, we have
\begin{equation}\label{23}
{h_{n-2}}{h_{n-1}}{h_{n+1}}{h_{n+2}}-h_n^4=Ed^2q^{n-2}\left[\left(p^2+q\right)h_n^2+Ed^2p^2q^{n-1}\right].
\end{equation}
\end{theorem}
\begin{proof}
For $ r=2 $ in \eqref{20}, we obtain
\begin{equation*}
{h_{n+2}}{h_{n-2}}-{h_{n}^2}=Ed^2p^2q^{n-2}.
\end{equation*}
Combining the last equation with Cassini's identity, we can write
\begin{align*}
{h_{n-2}}{h_{n-1}}&{h_{n+1}}{h_{n+2}}=\left(h_n^2+Ed^2q^{n-1}\right)\left(h_n^2+Ed^2p^2q^{n-2}\right)\\
&=h_n^4+\left(Ed^2q^{n-1}+Ed^2p^2q^{n-2}\right)h_n^2+Ed^2q^{n-1}Ed^2p^2q^{n-2}.
\end{align*}
The last equation completes the proof.
\end{proof}

The next theorem provides a number of sum formulas for the Horadam-Lucas numbers.
\begin{theorem}\label{K}
Let $ n $ be any integer. Then, we have
\begin{equation}\label{24}
\sum\limits_{i = 1}^n {{h_i}}  = \frac{{{h_{n + 1}} - q{h_n} - {h_1} + q{h_0}}}{{p - q - 1}},
\end{equation}
\begin{equation}\label{25}
\sum\limits_{i = 1}^n {\left(-1\right)^i{h_i}}  = \frac{\left(-1\right)^n\left( {h_{n + 1}} + q{h_n}\right) - {h_1} - q{h_0}}{{p + q + 1}},
\end{equation}
\begin{equation}\label{26}
\sum\limits_{i = 1}^n {{h_{2i}}}  = \frac{{{h_{2n + 2}} - {q^2}{h_{2n}} - {h_2} + {q^2}{h_0}}}{{{p^2} - {{\left( {q + 1} \right)}^2}}}
\end{equation}
and
\begin{equation}\label{27}
\sum\limits_{i = 1}^n {{h_{2i - 1}}}  = \frac{{{h_{2n + 1}} - {q^2}{h_{2n - 1}} - \left( {q + 1} \right){h_1} + pq{h_0}}}{{{p^2} - {{\left( {q + 1} \right)}^2}}}.
\end{equation}
\end{theorem}
\begin{proof}
To reduce the volume of paper, we prove only the first of these sum identities. The others can be proved by employing the same procedure. Let us denote the right-hand side of Eq. \eqref{24} by $ {a_n} $. By the definition of the Horadam-Lucas numbers, we obtain 
\begin{equation*}
{a_t} - {a_{t - 1}} = {h_t}.
\end{equation*}
Applying the idea of \enquote{creative telescoping} \cite{zei} to Eq. \eqref{24}, we conclude 
\begin{equation*}
\sum\limits_{i = 1}^n {{h_i}}  =\sum\limits_{t = 0}^n \left({{a_t} - {a_{t - 1}}}\right) ={a_n} - {a_{ - 1}},
\end{equation*}
and since $ {a_{ - 1}} = 0 $, the result follows.
\end{proof}

\section{Matrix Approach to Second-order sequences}
Note that the terms of the sequences in \eqref{I1} and \eqref{1} may also be stated as matrix recurrence relations. We define
\begin{equation}\label{28}
W_n=\left[ {\begin{array}{*{20}{c}}
	{w_{n+1}}&{w_n} \\ 
	{w_n}&{w_{n-1}} 
	\end{array}} \right]\,\,\text{and}\,\, H_n = \left[ {\begin{array}{*{20}{c}}
	{h_{n+1}}&{h_n} \\ 
	{h_n}&{h_{n-1}} 
	\end{array}} \right].
\end{equation}
Then we can write
\begin{equation}\label{29}
W_n = R W_{n-1}\,\,\text{and}\,\, H_n = R H_{n-1}.
\end{equation}
Extending the right-hand side of Eqs. \eqref{29} to zero, we obtain
\begin{equation}\label{30}
W_n = R^n W_{0}\,\,\text{and}\,\, H_n = R^n H_{0},
\end{equation}
where 
\begin{equation*}
W_0=\left[ {\begin{array}{*{20}{c}}
	{b}&{a} \\ 
	{a}&{\frac{pa-b}{q}} 
	\end{array}} \right]\,\,\text{and}\,\,H_0=\left[ {\begin{array}{*{20}{c}}
	{bp-2aq}&{2b-ap} \\ 
	{2b-ap}&{\frac{bp-a\left(p^2-2q\right)}{q}} 
	\end{array}} \right].
\end{equation*}
We then obtain the following theorem.
\begin{theorem}\label{L}
Let $ n $ be any integer. Then,
\begin{equation}\label{31}
h_n = \left(bp-2aq\right){u}_{n}-q\left(2b-ap\right){u}_{n-1},
\end{equation}
\begin{equation}\label{32}
h_n =bv_{n}-aqv_{n-1}.
\end{equation}
\end{theorem}
\begin{proof}
Eq. \eqref{30} gives Eq. \eqref{31}. From \cite{hor3}, we know
\begin{equation}\label{33}
v_n = 2{u}_{n+1}-p{u}_{n}=p{u}_{n}-2q{u}_{n-1}.
\end{equation}
Applying Eq. \eqref{33} to Eq. \eqref{31}, we obtain the second equation.
\end{proof}

Using Eq. \eqref{30}, we can also obtain the following theorem.
\begin{theorem}[Honsberger formula]\label{M} Let $ n $ and $ m $ be any integers. Then we have
\begin{equation}\label{34}
{w_{n + m}} = {u}_{n}w_{m}-q{u}_{n-1}w_{m-1}
\end{equation}
and
\begin{equation}\label{35}
{h_{n + m}} = {{u}_{m}}{h_{n + 1}} - q{{u}_{m-1}}{h_n}.
\end{equation} 
\end{theorem}
\begin{proof}
Replacing $ n + m $ with $ n $ in Eq. \eqref{30}, we can write
\begin{equation}\label{36}
W_{n+m+1}=R^{n+m+1}W_{0}=R^{n+1}W_{m}.
\end{equation}
The term $22$th of the matrix $ W_{n+m+1} $ is equal to the term $22$th of the product matrix, which gives the first result. Repeating the same method also provides the second equation.
\end{proof}
Note that in Theorem \ref{M}, symmetrical exchanges of $ n $ with $ m $ in each equation are possible.

\begin{theorem}\label{N}
For any integers $ n $ and $ k $, we can write
\begin{equation}\label{37}
{h_{n - k}} = {q^{1 - k}}\left( {{{u}_k}{h_{n - 1}} - {{u}_{k - 1}}{h_n}} \right) = {q^{ - k}}\left( {{h_n}{{u}_{k + 1}} - {h_{n + 1}}{{u}_k}} \right)
\end{equation}
and
\begin{equation}\label{38}
{w_{n - k}} = {q^{1 - k}}\left( {{{u}_k}{w_{n - 1}} - {{u}_{k - 1}}{w_n}} \right) = {q^{ - k}}\left( {{w_n}{{u}_{k + 1}} - {w_{n + 1}}{{u}_k}} \right).
\end{equation}
\end{theorem}
\begin{proof}
Using Eq. \eqref{30}, we obtain
\begin{equation*}
{H_{n - k}} = {R^{n - k}}{H_0} = {R^{ - k}}{R^n}{H_0} = {\left( {{R^k}} \right)^{ - 1}}{H_n}.
\end{equation*}
Using Eq. \eqref{21} after computing the inverse of $ {R^{ k}} $, we can write
\begin{equation*}
{H_{n - k}} = \frac{1}{{{q^k}}}\left[ {\begin{array}{*{20}{c}}
	{q\left( {{{u}_k}{h_n} - {{u}_{k - 1}}{h_{n + 1}}} \right)}&{q\left( {{{u}_k}{h_{n - 1}} - {{u}_{k - 1}}{h_n}} \right)} \\ 
	{{{u}_{k + 1}}{h_n} - {{u}_k}{h_{n + 1}}}&{{{u}_{k + 1}}{h_{n - 1}} - {{u}_k}{h_n}} 
	\end{array}} \right].
\end{equation*}
This completes the proof of  Eq. \eqref{37}. Eq. \eqref{38} can be proved similarly.
\end{proof}

From Theorems \ref{M} and \ref{N}, we obtain the following conclusion.
\begin{corollary}[Melham identity]\label{O}
Let $ n $ and $ k $ be any integers. Then,
\begin{equation}\label{39}
{h_{n + k + 1}}^2 - {q^{2k + 1}}{h_{n - k}}^2 = {d^2}{u_{2k + 1}}\left[ {\left( {{b^2} - {a^2}q} \right){u_{2n + 1}} - aq\left( {2b - ap} \right){u_{2n}}} \right].
\end{equation}
\end{corollary}
\begin{proof}
Considering Eqs. \eqref{35} and \eqref{37}, we can write
\begin{align*}
{h_{n + k + 1}}^2 - {q^{2k + 1}}{h_{n - k}}^2 &= {{u}_{k + 1}}^2{h_{n + 1}}^2 + {q^2}{{u}_k}^2{h_n}^2 - 2q{{u}_{k + 1}}{{u}_k}{h_{n + 1}}{h_n} \\
&- {q^{2k + 1}}{q^{ - 2k}}\left\{ {{h_n}^2{{u}_{k + 1}}^2 + {h_{n + 1}}^2{{u}_k}^2 - 2{{u}_{k + 1}}{{u}_k}{h_{n + 1}}{h_n}} \right\}\\
&= \left( {{{u}_{k + 1}}^2 - q{{u}_k}^2} \right)\left( {{h_{n + 1}}^2 - q{h_n}^2} \right)
\end{align*}
Applying Eq. \eqref{17} to the last equation, we obtain the claimed result.
\end{proof}

\begin{theorem}[General bilinear formula]\label{P}
Let $ a $, $ b $, $ c $, $ d $, and $ r $ be any integers satisfying $ a + b = c + d $. Then, we have
\begin{equation}\label{40}
u_{a}h_{b}-u_{c}h_{d} = q^r\left(u_{a-r}h_{b-r}-u_{c-r}h_{d-r}\right)
\end{equation}
and
\begin{equation}\label{41}
u_{a}w_{b}-u_{c}w_{d} = q^r\left(u_{a-r}w_{b-r}-u_{c-r}w_{d-r}\right).
\end{equation}
\end{theorem}
\begin{proof}
Employing the matrix equations in \eqref{28} and \eqref{29}, we obtain $ R^aH_b = R^cH_d$. Considering the entry $ \left(2,1\right) $ of the result, we can write
\begin{equation*}
u_{a}h_{b}-u_{c}h_{d} = q\left(u_{a-1}h_{b-1}-u_{c-1}h_{d-1}\right).
\end{equation*}
Repeating the same operations of $ r $ times yields Eq. \eqref{40}. The other result can be proved in a similar way.
\end{proof}

\end{document}